\documentclass[a4]{amsart}

\usepackage{amssymb}
\usepackage[all]{xy}
\usepackage{amsmath, amsfonts, amsthm , amssymb} 

\input xypic 
\usepackage{tikz} 

\newtheorem{theorem}{Theorem}[section]
\newtheorem{proposition}[theorem]{Proposition}
\newtheorem{lemma}[theorem]{Lemma}
\newtheorem{corollary}[theorem]{Corollary}

\theoremstyle{definition}
\newtheorem{definition}[theorem]{Definition}

\newtheorem{remark}[theorem]{Remark}

\newcommand{\uxa}{\ensuremath{(\underline{X},\underline{A})}} 
\newcommand{\cxx}{\ensuremath{(\underline{CX},\underline{X})}} 
\newcommand{\cxix}{\ensuremath{(CX,X)}}

\newcommand{\zk}{\ensuremath{\mathcal{Z}_{K}}} 
 
\newcommand{\LL}{\ensuremath{\mathcal{L}}} 

\newcommand{\hlgy}[1]{\ensuremath{H_{*}(#1)}}
 
\newcommand{\cohlgy}[1]{\ensuremath{H^{*}(#1)}} 
\newcommand{\rcohlgy}[1]{\ensuremath{\widetilde{H}^{*}(#1)}}

\newcounter{bean}
\newenvironment{letterlist}{\begin{list}{\rm ({\alph{bean}})}
      {\usecounter{bean}\setlength{\rightmargin}{\leftmargin}}}
      {\end{list}}

\newcommand{\namedright}[3]{\ensuremath{#1\stackrel{#2}
 {\longrightarrow}#3}}
\newcommand{\nameddright}[5]{\ensuremath{#1\stackrel{#2}
 {\longrightarrow}#3\stackrel{#4}{\longrightarrow}#5}}

\newcommand{\qqed}{\hfill\Box}

\begin{document}
\title{Moment-angle manifolds and Panov's problem}

\author{Stephen Theriault}
\address{School of Mathematics, University of Southampton, 
      Southampton SO17 1BJ, United Kingdom} 
\email{S.D.Theriault@soton.ac.uk}


\begin{abstract} 
We answer a problem posed by Panov, which is to describe the relationship 
between the wedge summands in a homotopy decomposition of the 
moment-angle complex corresponding to a disjoint union of $\ell$ points  
and the connected sum factors in a diffeomorphism decomposition of 
the moment-angle manifold corresponding to the simple polytope 
obtained by making $\ell$ vertex cuts on a standard $d$-simplex. 
This establishes a bridge between two very different approaches to 
moment-angle manifolds. 
\end{abstract} 

\maketitle 

\section{Introduction} 
\label{sec:intro} 

Moment-angle complexes have attracted a great deal of interest recently 
because they are a nexus for important problems arising in algebraic topology, 
algebraic geometry, combinatorics, complex geometry and commutative 
algebra. They are best described as a special case of the polyhedral 
product functor, popularized in~\cite{BBCG} as a generalization of 
moment-angle complexes and $K$-powers~\cite{BP2}, which were in  
turn generalizations of moment-angle manifolds~\cite{DJ}. 

Let $K$ be a simplicial complex on $m$ vertices.  For $1\leq i\leq m$,
let $(X_{i},A_{i})$ be a pair of pointed $CW$-complexes, where $A_{i}$ is 
a pointed subspace of $X_{i}$. Let $\uxa=\{(X_{i},A_{i})\}_{i=1}^{m}$ be 
the sequence of $CW$-pairs. For each simplex (face) $\sigma\in K$, let 
$\uxa^{\sigma}$ be the subspace of $\prod_{i=1}^{m} X_{i}$ defined by
\[\uxa^{\sigma}=\prod_{i=1}^{m} Y_{i}\qquad
       \mbox{where}\qquad Y_{i}=\left\{\begin{array}{ll}
                                             X_{i} & \mbox{if $i\in\sigma$} \\
                                             A_{i} & \mbox{if $i\notin\sigma$}.
                                       \end{array}\right.\] 
The \emph{polyhedral product} determined by \uxa\ and $K$ is
\[\uxa^{K}=\bigcup_{\sigma\in K}\uxa^{\sigma}\subseteq\prod_{i=1}^{m} X_{i}.\]
For example, suppose each $A_{i}$ is a point. If $K$ is a disjoint union
of $n$ points then $(\underline{X},\underline{\ast})^{K}$ is the wedge
$X_{1}\vee\cdots\vee X_{n}$, and if $K$ is the standard $(n-1)$-simplex
then $(\underline{X},\underline{\ast})^{K}$ is the product
$X_{1}\times\cdots\times X_{n}$. 

In the case when each pair of spaces $(X_{i},A_{i})$ equals $(D^{2},S^{1})$, 
the polyhedral product $\uxa^{K}$ is called a \emph{moment-angle complex}, 
and is written in more traditional notation as $\zk$. Two important properties 
of~$\zk$ are: the cohomology ring $\cohlgy{\zk;\mathbb{Z}}$ is the Tor-algebra 
$\mbox{Tor}_{\mathbb{Z}[v_{1},\ldots,v_{m}]}(\mathbb{Z}[K],\mathbb{Z})$ 
where $\mathbb{Z}[K]$ is the Stanley-Reisner face ring of~$K$ and 
$\vert v_{i}\vert=2$ for $1\leq i\leq m$; and $\zk$ is homotopy equivalent 
to the complement of the coordinate subspace arrangement in~$\mathbb{C}^{m}$ 
determined by $K$. Stanley-Reisner face rings are a subject of intense interest 
in commutative algebra (even having its own MSC number), and complements 
of coordinate subspace arrangements are an area of major importance 
in combinatorics. The connection to moment-angle complexes allows 
for topological methods to be used to inform upon problems posed in 
commutative algebra and combinatorics. 

To date, a great deal of work has been done to determine when $\zk$ 
is homotopy equivalent to a wedge of spheres, or to produce analogous 
statements in the case of certain polyhedral 
products~\cite{GT2,GT3,GPTW,GW,IK1,IK2}. When this is the case, 
the complement of the corresponding coordinate subspace arrangement 
is homotopy equivalent to a wedge of spheres, and $\cohlgy{\zk;\mathbb{Z}}$ 
is \emph{Golod}, meaning that all cup products and higher Massey products 
are zero. In all cases thus far, the arguments start by using combinatorics 
to identify a good  class of simplicial complexes to consider, then homotopy 
theory is used to prove that $\zk$ is homotopy equivalent to a wedge of 
spheres for this class of simplicial complexes, and finally it is deduced  
that $\cohlgy{\zk;\mathbb{Z}}$ is Golod. 

Moment-angle complexes arise in complex geometry and algebraic geometry 
in a different way. Let $P$ be a simple polytope, let $P^{\ast}$ be its dual, and 
let $\partial P^{\ast}$ be the boundary complex of $P^{\ast}$. Then 
$K=\partial P^{\ast}$ is a simplicial complex, and we let 
$\mathcal{Z}(P)=\mathcal{Z}_{\partial P^{\ast}}$. 
In this case, $\mathcal{Z}(P)$ has the richer structure of a manifold, and is called 
a \emph{moment-angle manifold}. These manifolds can be interpreted 
as intersections of complex quadrics, each fibring over a projective toric 
variety. The topology and geometry of these manifolds have been studied 
in considerable depth~\cite{BM,BP2,DJ,GL}. In particular, in~\cite{BM,GL} 
a large class of simple polytopes~$P$ was identified for which $\mathcal{Z}(P)$ 
is diffeomorphic to a connected sum of products of two spheres. 

Panov~\cite{P} observed that the two directions of work produce very similar 
results in the following way. If $K$ is a simplicial complex consisting of $\ell$ 
disjoint points, then by~\cite{GT1} there is a homotopy equivalence 
\[\zk\simeq\bigvee_{k=3}^{\ell+1}  (S^{k})^{\wedge (k-2)\binom{\ell}{k-1}}\] 
where $(S^{k})^{\wedge n}$ is the $n$-fold smash product of $S^{k}$ with itself.  
On the other hand, if $P$ is a simple polytope that has been obtained 
from the $d$-simplex by iteratively cutting off a vertex $\ell -1$ times 
(the cuts occuring in any order), then by~\cite{BM} there is a diffeomorphism 
\[\mathcal{Z}(P)\cong 
          \#_{k=3}^{\ell+1}  (S^{k}\times S^{\ell+2d-k})^{\#(k-2)\binom{\ell}{k-1}}\] 
where $(S^{k}\times S^{\ell +2d-k})^{\# n}$ is the $n$-fold connected sum 
of $S^{k}\times S^{\ell +2d-k}$ with itself. The coefficients and the sphere 
dimensions in both decompositions coincide. This led Panov to pose the following. 
\medskip 

\noindent 
\textbf{Problem}: Describe the nature of this correspondence. 
\medskip 

The purpose of the paper is to answer this problem, thereby establishing 
a bridge between two very different approaches to moment-angle manifolds. 
Let $P$ be a simple 
polytope obtained from a $d$-simplex by $\ell -1$ vertex cuts. To study 
polyhedral products, we  consider the dual simplicial complex~$P^{\ast}$, 
which is a \emph{stacked polytope} (defined explicitly in Section~\ref{sec:vertex}). 
We show that the homotopy type of $\mathcal{Z}_{\partial P^{\ast}}$ 
is independent of the stacking order for $P^{\ast}$ 
(dual to the result in~\cite{BM,GL} that the diffeomorphism type 
of $\mathcal{Z}(P)$ is independent of the order in which the vertex 
cuts occur for $P$). This lets us choose a stacking order, yielding a 
stacked polytope $\LL$ on the vertex set $[m]$ for $m=d+\ell$, 
which is more convenient to analyze (see Section~\ref{sec:deletion} 
for details). We prove the following. Let $\partial\LL-\{1\}$ be the 
full subcomplex of $\partial\LL$ obtained by deleting the vertex $\{1\}$. 

\begin{theorem} 
   \label{Panovsoln} 
   The stacked polytope $\LL$ has the following properties:  
   \begin{letterlist} 
      \item there is a homotopy equivalence   
                $\mathcal{Z}_{\partial\LL-\{1\}}\simeq\mathcal{Z}_{P_{\ell}}$ 
                where $P_{\ell}$ is $\ell$ disjoint points; 
      \item the inclusion 
                \(\namedright{\partial\LL-\{1\}}{}{\partial\LL}\) 
                induces a map 
                \(\namedright{\mathcal{Z}_{\partial\LL-\{1\}}}{}{\mathcal{Z}_{\partial\LL}}\), 
                which up to homotopy equivalences, is a map 
                \[f\colon\namedright 
                      {\bigvee_{k=3}^{\ell+1}  (S^{k})^{\wedge (k-2)\binom{\ell}{k-1}}} 
                      {}{\#_{k=3}^{\ell+1}  (S^{k}\times S^{\ell+2d-k})^{\#(k-2)\binom{\ell}{k-1}}};\] 
      \item $f$ has a left homotopy inverse $g$; 
      \item when restricted to a factor $\cohlgy{S^{k}\times S^{\ell+2d-k}}$ in 
               the cohomology of the connected sum, $f^{\ast}$ is zero on precisely 
               one of the ring generators. 
\end{letterlist} 
\end{theorem}   

It is helpful to point out one consequence of Theorem~\ref{Panovsoln}. Since $f$ has 
a left homotopy inverse, $f^{\ast}$ is an epimorphism. By part~(d), $f^{\ast}$ 
is nonzero on precisely one ring generator when restricted to any factor 
$\cohlgy{S^{k}\times S^{\ell+2d-k}}$. Let $A$ be the collection of such 
generators, one from each factor in the connected sum. The matching 
coefficients in the wedge decomposition of $\mathbb{Z}_{\partial\LL-\{1\}}$ 
and the connected sum decomposition of $\mathbb{Z}_{\partial\LL}$ then 
implies that $f^{\ast}$ maps $A$ isomorphically onto 
$\cohlgy{\bigvee_{k=3}^{\ell+1}  (S^{k})^{\wedge (k-2)\binom{\ell}{k-1}}}$. 

Along the way, we phrase as many of the intermediate results as possible 
in terms of polyhedral products $\cxx^{K}$, where $CX$ is the cone on a 
space $X$, or in terms of $\cxix^{K}$, where all the coordinate spaces $X_{i}$ 
equal a common space $X$. This is of interest because, when 
$K=\partial P^{\ast}$ for~$P$ a simple polytope obtained from a 
$d$-simplex by $\ell -1$ vertex cuts, $\cxix^{K}$ is analogous to the connected 
sum of products of two spheres. This analogue is \emph{not} a connected 
sum, in general, nor is it even a manifold. So understanding its homotopy 
theory helps distinguish how much of the homotopy theory of a connected 
sum depends on the actual geometry. 

The author would like to thank the referee for many helpful comments.

\section{Preliminary homotopy theory} 
\label{sec:prelim} 

In this section we give preliminary results regarding the homotopy 
theory of polyhedral products that will be used later on. In particular, 
in Proposition~\ref{htype} we identify a family of simplicial complexes whose 
polyhedral products have the same homotopy type as the polyhedral 
product corresponding to a disjoint union of points. 

For spaces $A$ and $B$, the \emph{right half-smash} $A\rtimes B$ 
is the space $(A\times B)/\sim$ where $(\ast,b)\sim\ast$. The 
\emph{join} $A\ast B$ is the space $(A\times I\times B)/\sim$ where 
$(a,1,b)\sim (\ast,1,b)$ and $(a,0,b)\sim (a,0,\ast)$; it is well known 
that there is a homotopy equivalence $A\ast B\simeq\Sigma A\wedge B$. 
The following lemma was proved in~\cite{GT2}. 

\begin{lemma} 
   \label{polemma} 
   Suppose that there is a homotopy pushout 
   \[\diagram 
           A\times B\rto^-{\pi_{1}}\dto^{\ast\times 1} & A\dto \\ 
           C\times B\rto & Q  
     \enddiagram\] 
   where $\pi_{1}$ is the projection onto the first factor. Then there 
   is a homotopy equivalence $Q\simeq (A\ast B)\vee (C\rtimes B)$.~$\qqed$ 
\end{lemma} 

Suppose that $K$ is a simplicial complex on the vertex set $\{1,\ldots,m\}$. 
If $L$ is a sub-complex of~$K$ on vertices $\{i_{1},\ldots,i_{k}\}$ then when 
applying the polyhedral product to $K$ and $L$ simultaneously, we must 
regard $L$ as a simplicial complex $\overline{L}$ on the vertices 
$\{1,\ldots,m\}$. By definition of the polyhedral product, we therefore 
obtain  
\[\cxx^{\overline{L}}=\cxx^{L}\times\prod_{t=1}^{m-k} X_{j_{t}}\] 
where $\{j_{1},\ldots,j_{m-k}\}$ is the complement of $\{i_{1},\ldots,i_{k}\}$ 
in $\{	1,\ldots,m\}$. 

The following lemma describes the homotopy type of $\cxx^{K}$ 
when $K=K_{1}\cup\Delta^{k}$, where $K_{1}$ and $\Delta^{k}$ 
have been glued along a common face $\Delta^{k-1}$. A similar 
gluing lemma was proved in~\cite{GT2} that was stated more generally in 
terms of two simplicial complexes joined along a common face, although 
it was stated only in the more restrictive case of 
$(\underline{C\Omega X},\underline{\Omega X})$. For our purposes,  
it is helpful to be more explicit about the vertices in $\Delta^{k-1}$, 
which affects the homotopy type of $\cxx^{K}$, so a proof is included. 

\begin{lemma} 
   \label{glue} 
   Let $K$ be a simplicial complex on the vertex set $\{1,\ldots,m\}$. 
   Suppose that $K=K_{1}\cup\Delta^{k}$ where: (i) $K_{1}$ is a simplicial 
   complex on the vertex set $\{1,\ldots,m-1\}$ and $\{i\}\in K_{1}$ for 
   $1\leq i\leq m-1$; (ii)~$\Delta^{k}$ is on the vertex set $\{m-k,\ldots,m\}$, 
   and (iii)~$K_{1}\cap\Delta^{k}$ is a $(k-1)$-simplex on the vertex set 
   $\{m-k,\ldots,m-1\}$. Then there is a homotopy equivalence 
   \[\cxx^{K}\simeq\bigg(\big(\prod_{i=1}^{m-k-1} X_{i}\big)\ast X_{m}\bigg)\vee 
       \bigg(\cxx^{K_{1}}\rtimes X_{m}\bigg).\] 
\end{lemma} 

\begin{proof} 
The simplicial complex $K$ can be written as a pushout 
\[\diagram 
          \Delta^{k-1}\rto\dto & \Delta^{k}\dto \\ 
          K_{1}\rto & K. 
  \enddiagram\] 
Regarding $K_{1}$, $\Delta^{k}$ and $\Delta^{k-1}$ as simplicial complexes on 
the vertex set $\{1,\ldots,m\}$ and applying the polyhedral product functor, 
we obtain a pushout 
\begin{equation} 
  \label{cxxpo} 
  \diagram 
         \cxx^{\overline{\Delta^{k-1}}}\rto\dto & \cxx^{\overline{\Delta^{k}}}\dto \\ 
         \cxx^{\overline{K_{1}}}\rto & \cxx^{K}. 
  \enddiagram 
\end{equation} 
We now identify the spaces and maps in~(\ref{cxxpo}). 

By hypothesis, $K_{1}$ is a simplicial complex on the vertex set 
$\{1,\ldots,m-1\}$, $\Delta^{k}$ is on the vertex set 
$\{m-k,\ldots,m\}$ and $\Delta^{k-1}$ is on the vertex set 
$\{m-k,\ldots,m-1\}$. So by definition of the polyhedral product we have 
\begin{align*} 
   \cxx^{\overline{\Delta^{k-1}}} & =\prod_{i=1}^{m-k-1} X_{i}\times 
       \prod_{i=m-k}^{m-1} CX_{i}\times X_{m} \\  
   \cxx^{\overline{\Delta^{k}}} & =\prod_{i=1}^{m-k-1} X_{i}\times 
       \prod_{i=m-k}^{m} CX_{i} \\ 
   \cxx^{\overline{K_{1}}} & =\cxx^{K_{1}}\times X_{m}. 
\end{align*} 
Further, under these identifications, the map 
\(\namedright{\cxx^{\overline{\Delta^{k-1}}}}{}{\cxx^{\overline{\Delta^{k}}}}\) 
is the identity on each factor indexed by $1\leq i\leq m-1$ and is the 
inclusion   
\(\namedright{X_{m}}{}{CX_{m}}\) 
on the $m^{th}$ factor, and the map 
\(\namedright{\cxx^{\overline{\Delta^{k-1}}}}{}{\cxx^{\overline{K_{1}}}}\) 
is the identity on the $m^{th}$ factor. Therefore, as the cone $CX_{i}$ is 
contractible, up to homotopy equivalences (\ref{cxxpo}) is the same as 
the homotopy pushout 
\begin{equation} 
   \label{cxxpo2} 
   \diagram 
        (\prod_{i=1}^{m-k-1} X_{i})\times X_{m}\rto^-{\pi_{1}}\dto^{f\times 1} 
            & \prod_{i=1}^{m-k-1} X_{i}\dto \\ 
        \cxx^{K_{1}}\times X_{m}\rto & \cxx^{K}  
   \enddiagram 
\end{equation} 
where $\pi_{1}$ is the projection and $f$ is some map. 

By~\cite{GT3}, any simplicial complex $L$ on vertices $\{1,\ldots,\ell\}$  
for which $\{i\}\in L$ for $1\leq i\leq\ell$ has the property that the inclusion 
\(\namedright{\prod_{i=1}^{\ell} X_{i}}{}{\cxx^{L}}\) 
is null homotopic. In our case, by hypothesis, $\{i\}\in K_{1}$ for 
$1\leq i\leq m-1$, so the inclusion 
\(\namedright{\prod_{i=1}^{m-1} X_{i}}{}{\cxx^{K_{1}}}\) 
is null homotopic. Since the map $f$ in~(\ref{cxxpo2}) factors 
through this inclusion, it too is null homotopic. Therefore 
Lemma~\ref{polemma} applies to the homotopy pushout~(\ref{cxxpo2}), 
giving a homotopy equivalence 
\[\cxx^{K}\simeq\bigg(\big(\prod_{i=1}^{m-k-1} X_{i}\big)\ast X_{m}\bigg)\vee 
       \bigg(\cxx^{K_{1}}\rtimes X_{m}\bigg).\] 
\end{proof} 

For example, let $P_{m}$ be $m$ disjoint points. Then 
$P_{m}=P_{m-1}\cup\Delta^{0}$ where $\Delta^{0}$ is a single point, 
and the union is taken over the emptyset. Applying Lemma~\ref{glue} 
then immediately gives the following.   

\begin{corollary} 
   \label{disjointpts} 
   There is a homotopy equivalence 
   \[\cxx^{P_{m}}\simeq\bigg(\big(\prod_{i=1}^{m-1} X_{i}\big)\ast X_{m}\bigg) 
       \vee\bigg(\cxx^{P_{m-1}}\rtimes X_{m}\bigg).\] 
   $\qqed$ 
\end{corollary} 

In Proposition~\ref{htype} we will consider the polyhedral product 
$\cxx^{K}$ where all the coordinate spaces $X_{i}$ are equal to a 
common space $X$. In this case, we write $\cxix^{K}$. In particular, 
in the case of $m$ disjoint points, Corollary~\ref{disjointpts} implies 
that there is a homotopy equivalence 
\begin{equation} 
   \label{disjptscxix} 
   \cxix^{P_{m}}\simeq\bigg(\big(\prod_{i=1}^{m-1} X\big)\ast X\bigg) 
       \vee\bigg(\cxix^{P_{m-1}}\rtimes X\bigg).  
\end{equation}    
          
\begin{proposition} 
   \label{htype} 
   Let $k\geq 1$ and suppose that there is a sequence of simplicial complexes 
   \[K_{1}=\Delta^{k}\subseteq K_{2}\subseteq\cdots\subseteq K_{\ell}\] 
   such that, for $i>1$, $K_{i}=K_{i-1}\cup_{\sigma_{i}}\Delta^{k}$ where 
   $\sigma_{i}=\Delta^{k-1}$. That is, $K_{i}$ is obtained from $K_{i-1}$ 
   by gluing on a $\Delta^{k}$ along the common face $\sigma_{i}$. Let 
   $K=K_{\ell}$ and observe that $K$ is a simplicial complex on $k+\ell$ 
   vertices. Then there is a homotopy equivalence 
   \[\cxix^{K}\simeq\cxix^{P_{\ell}}.\]     
\end{proposition} 

\begin{remark} It may be useful to note that Proposition~\ref{htype} 
also makes sense for $k=0$, in which case $\Delta^{0}$ is a point and 
each $\sigma_{i}$ is the emptyset, in which case $K=K_{\ell}$ is $\ell$ 
disjoint points, and the conclusion is a tautology. In the case when $k=1$, 
notice that $K=K_{\ell}$ is formed by iteratively taking an interval at 
stage $i$ and gluing one of its endpoints to a vertex of the preceeding 
simplicial complex at stage $i-1$. One example of this is the boundary of 
the $(\ell+2)$-gon with one vertex removed, another is all $\ell$ 
intervals joined at a common vertex.  
\end{remark} 

\begin{proof} 
Fix $k\geq 1$. The proof is by induction on $\ell$. 
When $\ell=1$, we have $K=K_{1}=\Delta^{k}$. By definition of 
the polyhedral product, 
$\cxx^{K}=\prod_{i=1}^{k+1} CX_{i}$, so 
$\cxix^{K}=\prod_{i=1}^{k+1} CX$. On the other hand, as $P_{1}$ is a single 
point, $\cxx^{P_{1}}=\cxix^{P_{1}}=CX$. Thus $\cxix^{K}\simeq\cxix^{P_{1}}$ as 
both spaces are contractible. 

Suppose that the proposition holds for all integers $t$ satisfying $t<\ell$. 
Consider $K_{\ell}=K_{\ell-1}\cup_{\sigma_{\ell}}\Delta^{k}$ where 
$\sigma_{\ell}=\Delta^{k-1}$. Reordering the vertices if necessary, 
we may assume that $K_{\ell-1}$ is a simplicial complex on the vertex 
set $\{1,\ldots,k+\ell -1\}$, $\Delta^{k}$ is on the vertex set 
$\{\ell,\ldots,k+\ell\}$, and $\sigma_{\ell}=\Delta^{k-1}$ is on the vertex 
set $\{\ell,\ldots,k+\ell -1\}$. By Lemma~\ref{glue}, there is a 
homotopy equivalence 
\[\cxx^{K_{\ell}}\simeq\bigg(\big(\prod_{i=1}^{\ell-1} X_{i}\big)\ast X_{k+\ell}\bigg) 
    \vee\bigg(\cxx^{K_{\ell-1}}\rtimes X_{k+\ell}\bigg).\] 
Therefore, there is a homotopy equivalence 
\[\cxix^{K_{\ell}}\simeq\bigg(\big(\prod_{i=1}^{\ell-1} X\big)\ast X\bigg) 
    \vee\bigg(\cxix^{K_{\ell-1}}\rtimes X\bigg).\] 
This formula is exactly the same as that in~(\ref{disjptscxix}) for 
$\cxix^{P_{\ell}}$. By inductive hypothesis, 
$\cxix^{K_{\ell -1}}\simeq\cxix^{P_{\ell -1}}$, 
so we obtain $\cxix^{K_{\ell}}\simeq\cxix^{P_{\ell}}$. The proposition 
therefore holds by induction. 
\end{proof}

\section{Vertex cuts and stacked polytopes} 
\label{sec:vertex} 

In this section we discuss some constructions obtained from 
simple polytopes, and discuss some of their properties in the 
context of polyhedral products. We begin with some definitions 
(see, for example,~\cite[Chapter 1]{BP2}). 

A \emph{(convex) polytope} is the convex hull of a finite set of points 
in $\mathbb{R}^{n}$. Its dimension is the dimension of its affine hull. 
Let $P$ be a $d$-dimensional polytope. A \emph{facet} of $P$ is 
a $(d-1)$-dimensional face. The polytope $P$ is \emph{simple} if 
each vertex lies in exactly $d$ facets of $P$. A partial ordering 
may be defined on the faces of $P$ by inclusion. This determines 
a poset called the \emph{face poset} of $P$. The opposite poset, 
given by reversing the order, determines another polytope $P^{\ast}$ 
called the \emph{dual} of $P$. If $P$ is simple then $P^{\ast}$ is a  
simplicial complex. Dualizing has the property that $P^{\ast\ast}=P$. 
Let $\partial P^{\ast}$ be the boundary of $P^{\ast}$. 

Suppose that $P$ is a simple polytope. Following~\cite{BP2,DJ}, a 
moment-angle complex $\mathcal{Z}(P)$ can be associated to $P$ by 
defining $\mathcal{Z}(P)=\mathcal{Z}_{\partial P^{\ast}}$. Generalizing to 
polyhedral products in the case where each coordinate space equals a 
common space $X$, define $\cxix(P)$ as $\cxix^{\partial P^{\ast}}$. 
The moment-angle complex $\mathcal{Z}(P)$ is in fact a manifold, 
but this property does not extend in general to $\cxix(P)$. 

An operation that produces new simple polytopes from existing 
ones is by doing vertex cuts. 

\begin{definition} 
   Let $P$ be a simple polytope of dimension $d$ and let 
   $V(P)$ be its vertex set. A hyperplane $H$ in $\mathbb{R}^{d}$ 
   cuts a vertex $x$ of $P$ if $x$ and $V(P)/\{x\}$ lie 
   in different open half-spaces of $H$. Let $Q$ be the intersection of 
   $P$ with the closed half-space of $H$ containing $V(P)/\{x\}$. 
   We say that~$Q$ is obtained from $P$ be a \emph{vertex cut operation}. 
\end{definition} 

Diagrammatically, this is pictured as follows: 
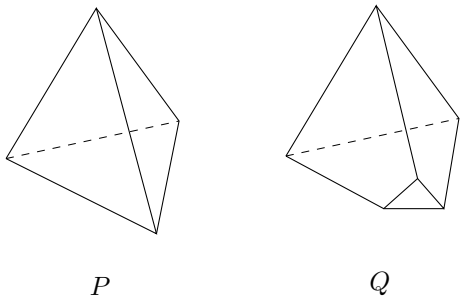
\[\begin{tikzpicture} 
     \draw (0,0)--(1.2,2)--(2,-1)--(0,0); 
     \draw (1.2,2)--(2.3,0.5)--(2,-1); 
     \draw [dashed, ultra thin] (0,0)--(2.3,0.5); 
     \node at (1.25,-1.7) {$P$}; 
  \end{tikzpicture}\qquad\qquad 
  \begin{tikzpicture} 
     \draw (0,0)--(1.2,2)--(1.75,-0.3)--(1.3,-0.7)--(0,0); 
     \draw (1.2,2)--(2.3,0.5)--(2.1,-0.7)--(1.75,-0.3);  
     \draw (1.3,-0.7)--(2.1,-0.7); 
     \draw [dashed, ultra thin] (0,0)--(2.3,0.5); 
     \node at (1.25,-1.7) {$Q$}; 
  \end{tikzpicture}\] 

The dual of a vertex cut operation is a stacking operation. 

\begin{definition} 
   Let $K$ be a simplicial complex of dimension $d$ and let $\sigma$ 
   be a facet of $K$. Define $L$ as $K\cup_{\sigma}\Delta^{d}$, that is, 
   $L$ is obtained from $K$ by gluing a $d$-simplex onto $K$ along the 
   facet $\sigma$. We say that $L$ is obtained from $K$ by a 
   \emph{stacking operation}. 
\end{definition} 

Diagrammatically, this is pictured as follows:   
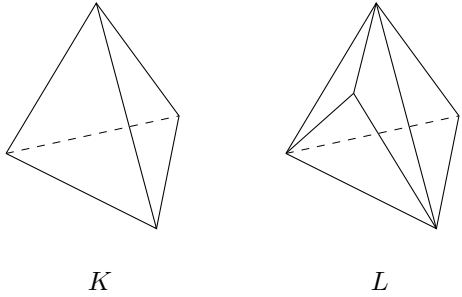
\[\begin{tikzpicture} 
     \draw (0,0)--(1.2,2)--(2,-1)--(0,0); 
     \draw (1.2,2)--(2.3,0.5)--(2,-1); 
     \draw [dashed, ultra thin] (0,0)--(2.3,0.5); 
     \node at (1.25,-1.7) {$K$}; 
  \end{tikzpicture}\qquad\qquad 
  \begin{tikzpicture} 
     \draw (0,0)--(1.2,2)--(2,-1)--(0,0); 
     \draw (1.2,2)--(2.3,0.5)--(2,-1); 
     \draw (0,0)--(0.9,0.8)--(1.2,2); 
     \draw (0.9,0.8)--(2,-1); 
     \draw [dashed, ultra thin] (0,0)--(2.3,0.5); 
     \node at (1.25,-1.7) {$L$}; 
  \end{tikzpicture}\] 

Notice that it is immediate from the definitions that the vertex cut 
and stacking operations preserve dimension. 

The objects we wish to study are the moment-angle manifold 
$\mathcal{Z}(P)$ and the polyhedral product $\cxix(P)$ where 
$P$ is a simple polytope obtained from $\Delta^{d}$ by iterated 
vertex cut operations. Equivalently, we study the polyhedral products 
$\mathcal{Z}_{\partial P^{\ast}}$ and $\cxix^{\partial P^{\ast}}$ 
where $\partial P^{\ast}$ is the boundary of a simple polytope 
obtained from $\Delta^{d}$ by iterated stacking operations. 

An important property of the vertex cut operation is that the diffeomorphism 
type of $\mathcal{Z}(P)$ is independent of the order in which the vertices 
were cut~\cite[Theorem 2.1]{GL}. Dually, the diffeomorphism type of 
$\mathcal{Z}(P)$ is independent of the stacking order for $P^{\ast}$. 
Weakening to homotopy type, we generalize this property to polyhedral 
products. 

\begin{proposition} 
   \label{cutindep} 
   Let $P$ be a simple polytope and let $Q$ be a simple polytope 
   obtained from $P$ by a vertex cut operation. Then the homotopy 
   type of $\cxix(Q)$ is independent of which vertex was~cut. 
\end{proposition} 

\begin{remark} 
It is easy to see that Proposition~\ref{cutindep} does not hold when 
$\cxix$ is replaced by $\cxx$, that is, when the coordinate spaces $X_{i}$ 
may be different. For example, let $P=\Delta^{2}$ with vertex set $\{1,2,3\}$. 
Cut vertex $1$ to obtain a new polytope $Q_{1}$ on the vertex set $\{2,3,4,5\}$ 
or cut vertex $2$ to obtain a new polytope $Q_{2}$ on the vertex set 
$\{1,3,4,5\}$. Both $Q_{1}$ and $Q_{2}$ equal the square~$I^{2}$. Notice 
that $Q_{1}$ and $Q_{2}$ are self-dual, so 
$\partial Q_{1}^{\ast}=\partial Q_{2}^{\ast}=\partial I^{2}$. 
Next, observe that $\partial I^{2}=A\ast B$ where $A$ and $B$ are $2$ 
points and $\ast$ is the join operation, defined in general by 
$K_{1}\ast K_{2}=\{\sigma_{1}\cup\sigma_{2}\mid \sigma_{i}\in K_{i}\}$. 
A straightforward property of the polyhedral product~\cite{BBCG} is 
that $\cxx^{K_{1}\ast K_{2}}\simeq\cxx^{K_{1}}\times\cxx^{K_{2}}$. 
In our case, this gives $\cxx^{\partial I^{2}}=\cxx^{A}\times\cxx^{B}$. 
Therefore, taking coordinate spaces $X_{i}$ for $1\leq i\leq 5$, we obtain 
$\cxx^{\partial Q_{1}^{\ast}}\simeq (X_{2}\ast X_{3})\times(X_{4}\ast X_{5})$ while 
$\cxx^{\partial Q_{2}^{\ast}}\simeq (X_{1}\ast X_{3})\times(X_{4}\ast X_{5})$. These 
have distinct homotopy types, but if each $X_{i}$ equals a common 
space $X$ then 
$\cxix(Q_{1})=\cxix^{\partial Q_{1}^{\ast}}\simeq 
     \cxix^{\partial Q_{2}^{\ast}}=\cxix(Q_{2})$. 
\end{remark} 

We will prove the equivalent, dual statement to Proposition~\ref{cutindep}. 

\begin{proposition} 
   \label{stackindep} 
   Let $K$ be a simplicial complex of dimension $d$ which is dual 
   to a simple polytope~$P$. Let $L$ be a simplicial complex obtained 
   from $K$ by stacking along a facet of $K$. Then the homotopy type 
   of $\cxix^{\partial L}$ is independent of which facet of $K$ was stacked. 
\end{proposition} 

\begin{proof} 
Let $\sigma_{1}$ and $\sigma_{2}$ be two facets of $K$. For $t=1,2$,  
let $\Delta^{d}_{t}$ be a $d$-simplex stacked onto $\sigma_{t}$. 
Then there are pushouts 
\[\diagram 
       \sigma_{t}\rto\dto & \Delta^{d}_{t}\dto \\ 
       K\rto & L_{t} 
  \enddiagram\] 
which define the simplicial complexes $L_{1}$ and $L_{2}$. Since 
$\sigma_{1}$ and $\sigma_{2}$ are faces of $\partial K$, the 
stacking operation also induces pushouts 
\[\diagram 
        \sigma_{t}\rto\dto & \partial\Delta^{d}_{t}\dto \\ 
        \partial K\rto & \partial L_{t}. 
  \enddiagram\] 
We will prove the proposition by showing that 
$\cxix^{\partial L_{1}}\simeq\cxix^{\partial L_{2}}$. 

It is useful to first consider $\cxx^{\partial L_{1}}$ and 
$\cxx^{\partial L_{2}}$ where we have to more explicitly keep track of 
coordinates. Suppose that the vertex set of $K$ is $\{1,\ldots,m\}$. 
Stacking introduces one additional vertex in $L_{t}$ which we label 
in both cases as $m+1$. Suppose that $\sigma_{t}$ is on the vertex set 
$\{i_{t,1},\ldots,i_{t,d}\}$. Let $\{j_{t,1},\ldots,j_{t,m-d}\}$ be the 
complement of $\{i_{t,1},\ldots,i_{t,d}\}$ in $\{1,\ldots,m\}$. Observe 
that $\Delta^{d}_{t}$ is on the vertex set $\{i_{t,1},\ldots,i_{t,d},m+1\}$. 
Regarding each of $\sigma_{t}$, $\Delta^{d}_{t}$, $K$ 
and $L_{t}$ as being on the vertex set $\{1,\ldots,m+1\}$, we can take 
polyhedral products to obtain pushouts 
\begin{equation} 
  \label{stackdgrm1} 
  \diagram 
       \cxx^{\overline{\sigma_{t}}}\rto\dto & \cxx^{\overline{\partial\Delta^{d}_{t}}}\dto \\ 
       \cxx^{\overline{\partial K}}\rto & \cxx^{\partial L_{t}}. 
  \enddiagram 
\end{equation} 

In general, if $\tau$ is a $d$-simplex on the vertex set \mbox{$\{1,\ldots,d+1\}$}, 
then by definition of the polyhedral product we have 
$\cxx^{\tau}=\prod_{r=1}^{d+1} CX_{r}$, 
and 
$\cxx^{\partial\tau}=\bigcup_{r=1}^{d+1}(CX_{1}\times\cdots\times X_{r}\times 
     \cdots\times CX_{d+1})$, 
where, in each term of the union, all the factors are cones except for one. 
Applying this to our case, we obtain  
\begin{align*} 
     \cxx^{\overline{\sigma_{t}}} & =\cxx^{\sigma_{t}}\times\prod_{s=d+1}^{m} X_{j_{t,s}} 
          \times X_{m+1}= 
          \prod_{s=1}^{d} CX_{i_{t,s}}\times\prod_{s=d+1}^{m} X_{j_{t,s}}\times X_{m+1} \\ 
     \cxx^{\overline{\partial\Delta^{d}_{t}}} & 
           = \cxx^{\partial\Delta^{d}_{t}}\times\prod_{s=d+1}^{m} X_{j_{t,s}}= 
           \bigcup_{s=1}^{d+1}(CX_{i_{t,1}}\times\cdots\times X_{i_{t,s}}\times\cdots 
           \times CX_{i_{t,d+1}})\times\prod_{s=d+1}^{m} X_{j_{t,s}} \\ 
     \cxx^{\overline{K}} & =\cxx^{K}\times X_{m+1}  
\end{align*} 
where, in the second line, to compress notation we have used $i_{t,d+1}$ to 
refer to the vertex $m+1$ for both $t=1,2$. 
Further, under these identifications, the map 
\(\namedright{\cxx^{\overline{\sigma_{t}}}}{}{\cxx^{\overline{\partial K}}}\) 
is the product of the identity map on $X_{m+1}$ and a map 
\(f_{t}\colon\namedright{\cxx^{\sigma_{t}}\times\prod_{s=d+1}^{m} X_{j_{t,s}}}{} 
         {\cxx^{\partial K}}\) 
induced by the inclusion of the face 
\(\namedright{\sigma_{t}}{}{\partial K}\), 
and the map 
\(\namedright{\cxx^{\overline{\sigma_{t}}}}{}{\cxx^{\overline{\partial\Delta^{d}}}}\) 
is a coordinate-wise inclusion which we label as $i$. 
Thus~(\ref{stackdgrm1}) can be identified with the pushouts 
\begin{equation} 
  \label{stackdgrm2} 
  \diagram 
       \bigg(\prod_{s=1}^{d} CX_{i_{t,s}}\times\prod_{s=d+1}^{m} X_{j_{t,s}}\bigg) 
                 \times X_{m+1}\rto^-{i}\dto^{f_{t}\times 1} 
           & \bigcup_{s=1}^{d+1}(CX_{i_{t,1}}\times\cdots\times X_{i_{t,s}}\times 
                       \cdots\times CX_{i_{t,d+1}})\dto \\ 
       \cxx^{\partial K}\times X_{m+1}\rto & \cxx^{\partial L_{t}}. 
  \enddiagram 
\end{equation} 
Now simplifying to the case of $\cxix$ when each coordinate space $X_{i}$ 
equals a common space $X$, we obtain pushouts 
\begin{equation} 
  \label{stackdgrm3} 
  \diagram 
       \bigg(\prod_{s=1}^{d} CX\times\prod_{s=d+1}^{m} X\bigg) 
              \times X\rto^-{i}\dto^{f_{t}\times 1} 
           & \bigcup_{s=1}^{d+1}(CX\times\cdots\times X\times\cdots\times CX)\dto \\ 
       \cxix^{\partial K}\times X\rto & \cxix^{\partial L_{t}}. 
  \enddiagram 
\end{equation}       

Observe that the only difference in the pushouts for $\cxix^{\partial L_{1}}$ and 
$\cxix^{\partial L_{2}}$ in~(\ref{stackdgrm3}) are the maps $f_{1}$ and $f_{2}$. 
We will show that there is a self-homotopy equivalence $e$ of $\cxix^{\partial K}$ 
which satisfies a homotopy commutative square 
\begin{equation} 
  \label{stackdgrm4} 
  \diagram 
         \prod_{s=1}^{d} CX\times\prod_{s=d+1}^{m} X\rto^-{p}\dto^{f_{1}} 
              & \prod_{s=1}^{d} CX\times\prod_{s=d+1}^{m} X\dto^{f_{2}} \\ 
         \cxix^{\partial K}\rto^-{e} & \cxix^{\partial K} 
  \enddiagram 
\end{equation} 
where $p$ permutes coordinates. Granting this, observe that we obtain 
a map from the $t=1$ pushout in~(\ref{stackdgrm3}) to the $t=2$ pushout 
by using $p\times 1$ on the upper left corner, $e\times 1$ on the lower 
right corner, and noting that $i$ is a coordinate-wise inclusion, we can 
also use $p\times 1$ on the upper right corner. This induces a map 
of pushouts 
\(h\colon\namedright{\cxix^{\partial L_{1}}}{}{\cxix^{\partial L_{2}}}\). 
As~$p$ and $e$ are homotopy equivalences, so is $h$, and this completes 
the proof. 

It remains to construct the self-homotopy equivalence $e$  
of $\cxix^{\partial K}$. First consider the simple polytope $P$ that is dual to $K$. 
Let $v_{1}$ and $v_{2}$ be vertices of $P$. Consider the permutation that 
interchanges $v_{1}$ and $v_{2}$ while leaving the other vertices fixed. 
Since the polytope $P$ is simple, this permutation induces a self-map   
of the face poset of $P$ which permutes the $k$-dimensional faces 
for each $0\leq k\leq d$. Dually, the face poset for $K$ is obtained by 
reversing the arrows on the face poset for $P$, so we obtain a self-map 
of the face poset of $K$ which permutes the $k$-dimensional faces 
for each $0\leq k\leq d$ . Consequently, if we let $v_{1}$ and $v_{2}$ 
be the vertices of $P$ that are dual to the facets $\sigma_{1}$ and 
$\sigma_{2}$ of $K$, we obtain a map 
\(g^{\prime}\colon\namedright{K}{}{K}\) 
of simplicial complexes which permutes the facets $\sigma_{1}$ 
and~$\sigma_{2}$. This induces a map 
\(g\colon\namedright{\partial K}{}{\partial K}\) 
of simplicial complexes which permutes the faces $\sigma_{1}$ 
and $\sigma_{2}$. Now apply the polyhedral product $\cxix$ to the 
face poset of $K$. Any face $\tau$ of $K$ has $\cxix^{\tau}$ equal 
to a product of copies of $CX$ or $X$, depending on whether a 
vertex is in or not in $\tau$. So the self-map of the face poset of $K$ 
induces a self-map of $\cxx^{\tau}$ for each face $\tau$ of $K$ 
which permutes the $CX$ factors and permutes the $X$ factors. 
Any such permutation is a homotopy equivalence. The morphism 
of face posets ensures that these permuations are compatible 
under face-wise inclusions, so there are induced maps 
\(e^{\prime}\colon\namedright{\cxix^{K}}{}{\cxix^{K}}\) 
and 
\(e\colon\namedright{\cxix^{\partial K}}{}{\cxix^{\partial K}}\) 
which are homotopy equivalences, and $e$ satisfies~(\ref{stackdgrm4}). 
\end{proof} 

Starting with a simplicial complex $K$ of dimension $d$, there are many 
ways of iteratively stacking to produce a new simplicial complex $L$. 
A particular sequence of stacks is called a \emph{stack history} of~$L$. 

\begin{corollary} 
   \label{stackcor} 
   Let $K$ be a simplicial complex of dimension $d$ which is dual 
   to a simple polytope~$P$ and let $L$ be a simplicial complex obtained 
   from $K$ by iterated stacking operations. Then the homotopy type 
   of $\cxix^{\partial L}$ is independent of the stack history of $L$.~$\qqed$ 
\end{corollary}

\section{Deleting a vertex from the boundary of a stacked polytope} 
\label{sec:deletion} 

In this section we consider a special case of iterated stacking operations. 
Let $P=\Delta^{d}$ be the $d$-simplex. Then $P$ is a simple polytope, 
and the dual $K=P^{\ast}$ of $P$ is again $\Delta^{d}$. In this case, 
if $L$ is obtained from $K$ by a sequence of stacking operations, then 
$L$ is also a simple polytope of dimension~$d$, as well as a simplicial 
complex. The simple polytope $L$ is called a \emph{stacked polytope}. 
Each copy of $\Delta^{d}$ in $L$ is called a \emph{stack}, so if $L$ is 
formed by $\ell -1$ stacking operations, then it has $\ell$ stacks. 

Suppose that $L$ is a stacked polytope with $\ell$ stacks. So there is 
a sequence of stacked polytopes 
\[L_{1}=\Delta^{d}\subseteq L_{2}\subseteq\cdots\subseteq L_{\ell}=L\] 
where, for $2\leq i\leq\ell$, $L_{i}$ has been formed by gluing a $\Delta^{d}$ 
to $L_{i-1}$ along a common facet. By Corollary~\ref{stackcor}, the homotopy 
type of $\cxix^{\partial L}$ is independent of the stack history of $L$. 
Thus we can choose a stacking order which is more convenient for 
analyzing $\cxix^{\partial L}$. 
 
The prescribed stacking order we choose is as follows. Let 
$\LL_{1}=\Delta^{d}$. Label the vertices of $\Delta^{d}$ as 
$\{1,\ldots,d+1\}$. Form $\LL_{2}$ by stacking a copy of $\Delta^{d}$ 
to $\LL_{1}$ on the facet $(1,\ldots,d)$. Label the one extra vertex 
of $\LL_{2}$ as $d+2$, and notice that if $d=1$ then the vertex $\{3\}$ 
is a facet of $\LL_{2}$ and if $d>1$ then $(1,\ldots,d-1,d+2)$ is a 
facet of $\LL_{2}$. Now stack onto this facet and iterate the procedure. 
We obtain, for $2< k\leq\ell$, a stacked polytope $\LL_{k-1}$ on the 
vertex set $\{1,\ldots,d+k-1\}$ where if $d=1$ then the vertex 
$\{k=d+k-1\}$ is a fact and if $d>1$ then $(1,\ldots,d-1,d+k-1)$ is 
a facet. Form $\LL_{k}$ by stacking a copy of $\Delta^{d}$ on this facet. 
Label the one extra vertex of $\LL_{k}$ as $d+k$, and observe that 
if $d=1$ then the vertex $\{d+k\}$ is a facet of $\LL_{k}$ and if $d>1$ then 
$(1,\ldots,d-1,d+k)$ is a facet of~$\LL_{k}$. Finally, let $\LL=\LL_{\ell}$. 

Now we identify the simplicial complex obtained by deleting the 
vertex $\{1\}$ from $\partial\LL$. 

\begin{lemma} 
   \label{delete} 
   The simplicial complex $\LL=\LL_{\ell}$ has the following properties:  
   \begin{letterlist} 
      \item $\LL$ has $\ell d+1$ facets; 
      \item there are $\ell$ facets in part~(a) which do not contain the vertex $\{1\}$: 
                these are $(2,3,\ldots,d,d+1)$, $(2,3,\ldots,d,d+2)$, and for 
                $2< k\leq\ell$, $(2,3,\ldots,d-1,d+k-1,d+k)$; 
      \item the simplicial complex $\partial\LL-\{1\}$ filters as a sequence 
                \[M_{1}=\Delta^{d-1}\subseteq M_{2}\subseteq\cdots\subseteq 
                        M_{\ell}=\partial L-\{1\}\] 
                where, for $2\leq k\leq\ell$, $M_{k}=M_{k-1}\cup_{\sigma_{k}}\Delta^{d-1}$, 
                where $\sigma_{k}=\Delta^{d-2}$. 
   \end{letterlist} 
\end{lemma} 

\begin{proof} 
For part~(a), observe that $\LL_{1}=\Delta^{d}$ has $d+1$ facets. As $\LL_{\ell}$ is formed 
by gluing on $(\ell -1)$ more $\Delta^{d}$'s, the total of $\ell$ copies of $\Delta^{d}$  
have $\ell(d+1)$ facets. But each gluing occurs along a common facet, 
so at each of the $(\ell -1)$ gluings $1$ facet is removed. Thus $\LL_{\ell}$ 
has $\ell(d+1)-(\ell-1)=\ell d+1$ facets. 

For part~(b), observe that in $\LL_{1}=\Delta^{d}$ there are $d+1$ facets but 
only one of them, $(2,3,\ldots,d,d+1)$, does not contain the vertex $\{1\}$. 
In forming $\LL_{2}$, we stack on the facet $(1,2,\ldots,d)$ of $\LL_{1}$, and label 
the extra vertex $d+2$. This operation removes $(1,2,\ldots,d)$ as a facet 
of $\LL_{1}$ and introduces $d$ new facets: all $d+1$ facets of $\Delta^{d}$ on 
the vertex set $\{1,2,\ldots,d,d+2\}$ except for $(1,2,\ldots,d)$. Of the new facets, 
only one of them, $(2,3,\ldots,d,d+2)$, does not contain the vertex $\{1\}$. 
Iterating, for $2<k\leq\ell$, in forming $\LL_{k}$, we stack on the facet 
$(1,2,\ldots,d-1,d+k-1)$ of $\LL_{k-1}$, and label the extra vertex $d+k$. 
This operation removes $(1,2,\ldots,d-1,d+k-1)$ as a facet of $\LL_{k}$ 
and introduces $d$ new facets: all $d+1$ facets of $\Delta^{d}$ on the vertex 
set $\{1,2,\ldots,d-1,d+k-1,d+k\}$ except for $(1,2,\ldots,d-1,d+k-1)$. 
Of the new facets, only one of them, $(2,3,\ldots,d-1,d+k-1,d+k)$, 
does not contain the vertex $\{1\}$. Thus, precisely $\ell$ of the 
$\ell d+1$ total facets of $\LL_{\ell}$ do not contain the 
vertex $\{1\}$, and these are: $(2,3,\ldots,d,d+1)$, $(2,3,\ldots,d,d+2)$, 
and for $2< k\leq\ell$, $(2,3,\ldots,d-1,d+k-1,d+k)$. 

For part~(c), since $\LL$ is a simple polytope which is also a 
simplicial complex, the geometric realization of $\partial\LL$  
can be obtained by gluing together the facets of $\LL$. The 
geometric realization of the simplicial complex $\partial\LL-\{1\}$ 
is therefore obtained by gluing together those facets of $\LL$ 
which do not contain the vertex $\{1\}$. We perform this gluing 
procedure one simplex at a time. Let $M_{1}=\Delta^{d-1}$ 
be $(2,3,\ldots,d,d+1)$. Form $M_{2}$ by gluing the $(d-1)$-simplex 
$(2,3,\ldots,d,d+2)$ to $M_{1}$ along the common $(d-2)$-simplex 
$(2,3,\ldots,d)$. For $2\leq k\leq\ell$, form $M_{k}$ by gluing the 
$(d-1)$-simplex $(2,3,\ldots,d-1,d+k-1,d+k)$ to $M_{k-1}$ along 
the common $(d-2)$-simplex $(2,3,\ldots,d-1,d+k-1)$. Then   
$M_{\ell}=\partial\LL-\{1\}$. 
\end{proof} 

Applying Proposition~\ref{htype} to Lemma~\ref{delete}~(c), we immediately 
obtain the following. 

\begin{proposition} 
   \label{deletiontype} 
   There is a homotopy equivalence 
   \[\cxix^{\partial\LL-\{1\}}\simeq\cxix^{P_{\ell}}\] 
   where $P_{\ell}$ is $\ell$ disjoint points.~$\qqed$ 
\end{proposition} 

Now specialize to polyhedral products on the pairs $(D^{2},S^{1})$ 
and write $\zk$ for $\cxix^{K}$. In~\cite{GT1} the homotopy type of 
$\mathcal{Z}_{P_{\ell}}$ was identified, giving the following. 

\begin{corollary} 
   \label{deletioncor} 
   There is a homotopy equivalence 
   \[\mathcal{Z}_{\partial\LL-\{1\}}\simeq\mathcal{Z}_{P_{\ell}}\simeq  
         \bigvee_{k=3}^{\ell+1}  (S^{k})^{\wedge (k-2)\binom{\ell}{k-1}}.\] 
   $\qqed$ 
\end{corollary}

\section{Cup products in $\cohlgy{\mathbb{Z}_{\partial\LL}}$} 
\label{sec:cup} 

On the one hand, since $\LL$ is a stacked polytope of dimension $d$ 
with $\ell$ stacks, it is dual to a simple polytope obtained from $\Delta^{d}$ 
by $\ell -1$ vertex cuts. So by~\cite{BM,GL} there is a diffeomorphism 
$\mathcal{Z}_{\partial\LL}\cong 
          \#_{k=3}^{\ell+1}  (S^{k}\times S^{\ell+2d-k})^{\#(k-2)\binom{\ell}{k-1}}$. 
The cup products in $\cohlgy{\mathcal{Z}_{\partial\LL}}$ are then clear 
from the description of the space as a connected sum of products of spheres. 
On the other hand, there is a combinatorial description of the cup product 
structure in $\mathcal{Z}_{K}$ for any simplicial complex $K$, proved 
in~\cite{BBP,BP1,F}. Take homology with integer coefficients. The \emph{join} 
of two simplicial complexes $K_{1}$ and $K_{2}$ is 
$K_{1}\ast K_{2}=\{\sigma_{1}\cup\sigma_{2}\mid \sigma_{i}\in K_{i}\}$. 

\begin{theorem}  
   \label{zkcohlgy}
   There is an isomorphism of graded commutative algebras
   \[\cohlgy{\zk}\cong\bigoplus_{I\subset[m]}\rcohlgy{K_{I}}.\] 
   Here, $\rcohlgy{K_{I}}$ denotes the reduced simplicial cohomology
   of the full subcomplex $K_{I}\subset K$ (the restriction of $K$
   to $I\subset[m]$). The isomorphism is the sum of isomorphisms
   \[H^p(\zk)\cong\sum_{I\subset[m]}\widetilde{H}^{p-|I|-1}(K_I)\]
   and the ring structure (the Hochster ring) is given by the maps
   \[\namedright{H^{p-|I|-1}(K_{I})\otimes H^{q-|J|-1}(K_{J})}{} 
            {H^{p+q-|I|-|J|-1}(K_{I\cup J})}\]
   which are induced by the canonical simplicial maps 
   \(\namedright{K_{I\cup J}}{}{K_{I}*K_{J}}\) 
   for $I\cap J=\emptyset$ and zero otherwise.~$\qqed$  
\end{theorem} 

Theorem~\ref{zkcohlgy} implies that the Hochster ring  structure on 
$\mathcal{Z}_{\partial\LL}$ matches the ring product  structure arising 
from the geometry of the connected sum, at least up to an isomorphism. 
We need information from both, so we are led to geometrically realize 
the isomorphism, via a homotopy equivalence. 

In general, if $M$ is an $n$-dimensional manifold, let $M-\ast$ be 
$M$ with a point in the interior of the $n$-disc removed. As a 
$CW$-complex, $M-\ast$ is homotopy equivalent to the $(n-1)$-skeleton 
of $M$. By definition of the connected sum, if $M$ and $N$ are two 
$n$-dimensional manifolds then  $(M\#N)-\ast\simeq (M-\ast)\vee (N-\ast)$. 
In our case, as 
$\mathcal{Z}_{\partial\LL}= 
    \#_{k=3}^{\ell+1}  (S^{k}\times S^{\ell+2d-k})^{\#(k-2)\binom{\ell}{k-1}}$, 
the $(\ell+2d-1)$-skeleton of $\mathcal{Z}_{\partial\LL}$ is the wedge 
\[W=\bigvee_{k=3}^{\ell+1}  (\bigvee_{t=1}^{(k-1)\binom{\ell}{k-1}} 
    (S^{k}\vee S^{\ell+2d-k})).\]  
Therefore, there is one ring generator in $\cohlgy{\mathcal{Z}_{\partial\LL}}$ 
for each sphere in the wedge $W$. 

Applying Theorem~\ref{zkcohlgy} to $\mathcal{Z}_{\partial\LL}$ we obtain  
an abstract isomorphism of algebras 
\(h\colon\namedright{\cohlgy{\mathcal{Z}_{\partial\LL}}}{} 
      {\cohlgy{\mathcal{Z}_{\partial\LL}}}\), 
where on the left the generating set is given by  the Hochster ring structure, 
on the right the generating set is given by the $CW$-structure of 
$\mathcal{Z}_{\partial\LL}$, and $h$ maps generators to 
generators. Restricting to degrees less than $\ell+2d$, we obtain an 
abstract isomorphism of modules 
\(h'\colon\namedright{\cohlgy{W}}{}{\cohlgy{W}}\). 
Dualizing,  we obtain an abstract isomorphism of modules 
\(h''\colon\namedright{\hlgy{W}}{}{\hlgy{W}}\). 
Since~$W$ is a wedge of spheres, the abstract map $h''$ may be 
realized geometrically, as follows. Let $n$ be the number of spheres 
in the wedge $W$ and label the spheres from $1,\ldots, n$. For 
$1\leq i\leq n$, let 
\(j_{i}\colon\namedright{S_{n}}{}{W}\) 
be the inclusion into the wedge, and let $x_{i}\in\hlgy{W}$ be the Hurewicz 
image of $j_{i}$. Suppose that 
$h''(x_{i})=t_{i,1}x_{1}+\cdots + t_{i,n}x_{n}$ for some integers $t_{i,1},\ldots,t_{i,n}$. 
Define 
\(g_{i}\colon\namedright{S_{i}}{}{W}\) 
by $g_{i}=t_{i,1}j_{1}+\cdots t_{i,n}j_{n}$. Let 
\(g\colon\namedright{W}{}{W}\) 
be the wedge sum of the maps $g_{i}$ for $1\leq i\leq n$. Then 
$g_{\ast}=h''$. Dualizing, $g^{\ast}=h'$. As $h'$ is an isomorphism, 
by Whitehead's Theorem $g$ is a homotopy equivalence. 
 
Next, the map attaching the top cell to $W$ to form $\mathcal{Z}_{\partial\LL}$ is 
a sum of Whitehead products, one Whitehead product for each 
$S^{k}\times S^{\ell+2d-k}$. This Whitehead product is detected in 
cohomology by a nonzero cup product. Since Theorem~\ref{zkcohlgy} 
gives a ring isomorphism between the cup product structures on 
$\cohlgy{\mathcal{Z}_{\partial\LL}}$ from the connected sum and the Hochster ring, 
$g$ can be extended to a map 
\[\Gamma\colon\namedright{\mathcal{Z}_{\partial\LL}}{}{\mathcal{Z}_{\partial\LL}}\] 
which induces an isomorphism in cohomology and so is a homotopy 
equivalence. Thus we have the following. 

\begin{lemma} 
   \label{cupprods} 
   Altering 
   $\mathcal{Z}_{\partial\LL}\cong
          \#_{k=3}^{\ell+1}  (S^{k}\times S^{\ell+2d-k})^{\#(k-2)\binom{\ell}{k-1}}$ 
   by a self-homotopy equivalence if necessary, we may assume that each 
   Hochster ring generator in $\cohlgy{\mathcal{Z}_{\partial\LL}}$ is represented 
   by a map 
   \(\namedright{S^{t}}{}{\#_{k=3}^{\ell+1}  (S^{k}\times S^{\ell+2d-k})^{\#(k-2)\binom{\ell}{k-1}}}\) 
   which is the inclusion of one of the spheres in the $(\ell+2d-1)$-skeleton of 
   the connected sum.~$\qqed$ 
\end{lemma} 

Lemma~\ref{cupprods} lets us use combinatorial information from 
the Hochster ring to deduce cup product information for the cohomology 
of the connected sum. We apply this to deduce some cup product 
information in $\cohlgy{\mathcal{Z}_{\partial\LL}}$.  

Let $\mathcal{I}$ be an index set which runs over 
all the products of two spheres in the connected sum 
$\#_{k=3}^{\ell+1}  (S^{k}\times S^{\ell+2d-k})^{\#(k-2)\binom{\ell}{k-1}}$. 
There are $\Sigma_{k=3}^{\ell+1}  (k-1)\binom{\ell}{k-1}$ 
elements in $\mathcal{I}$. Each $\alpha\in\mathcal{I}$ corresponds 
to a product of spheres $S^{k}\times S^{\ell +2d-k}$ which  
determines a nontrivial cup product in $\cohlgy{\mathcal{Z}_{\partial\LL}}$: 
if $x_{\alpha},y_{\alpha}\in\cohlgy{\mathcal{Z}_{\partial\LL}}$ are generators 
corresponding to the inclusions of $S^{k}$ and $S^{\ell+2d-k}$ into 
the $(\ell+2d-1)$-skeleton of the connected sum, then 
$x_{\alpha}\cup y_{\alpha}\neq 0$. By Lemma~\ref{cupprods}, we 
may assume that $x_{\alpha}$ and $y_{\alpha}$ are Hochster ring 
generators. Thus $x_{\alpha}\in\rcohlgy{\partial\LL_{I_{\alpha}}}$ 
and $y_{\alpha}\in\rcohlgy{\partial\LL_{J_{\alpha}}}$ for some 
index sets $I_{\alpha}$ and $J_{\alpha}$ of $[m]$, where $m=\ell+d$ 
is the number of vertices of $\partial\LL$. 

The inclusion 
\(\namedright{\partial\LL-\{1\}}{}{\partial\LL}\) 
induces a map 
\[f\colon\namedright{\mathcal{Z}_{\partial\LL-\{1\}}}{}{\mathcal{Z}_{\partial\LL}}.\]   

\begin{lemma} 
   \label{cupseparate} 
   $\mbox{Ker}\, f^{\ast}$ contains one and only one of $x_{\alpha}$ or $y_{\alpha}$. 
\end{lemma} 

\begin{proof} 
First, in the Hochster ring for $\cohlgy{\mathcal{Z}_{\partial\LL}}$ we have 
$x_{\alpha}\cup y_{\alpha}\neq 0$. So by Theorem~\ref{zkcohlgy},  
$I_{\alpha}\cap J_{\alpha}=\emptyset$. 

We claim that $1\in I_{\alpha}\cup J_{\alpha}$. For if not, 
then $I_{\alpha}\cup J_{\alpha}$ is contained 
in the vertex set for $\partial\LL -\{1\}$. Therefore, by 
Theorem~\ref{zkcohlgy} all three of 
$\rcohlgy{\partial\LL_{I_{\alpha}}},\rcohlgy{\partial\LL_{J_{\alpha}}}, 
     \rcohlgy{\partial\LL_{I_{\alpha}\cup J_{\alpha}}}$   
are contained in $\cohlgy{\mathcal{Z}_{\partial\LL-\{1\}}}$. 
That is, $x_{\alpha},y_{\alpha}\in\cohlgy{\mathcal{Z}_{\partial\LL-\{1\}}}$ 
and so $x_{\alpha}\cup y_{\alpha}\in\cohlgy{\mathcal{Z}_{\partial\LL-\{1\}}}$. 
But by Corollary~\ref{deletioncor}, $\mathcal{Z}_{\partial\LL-\{1\}}$ 
is homotopy equivalent to a wedge of spheres, implying that all 
the cup products in its cohomology are zero, a contradiction. 

Now, since $1\in I_{\alpha}\cup J_{\alpha}$ and 
$I_{\alpha}\cap J_{\alpha}=\emptyset$, either $1\in I_{\alpha}$ 
or $1\in J_{\alpha}$. If $1\in I_{\alpha}$ then $1\notin J_{\alpha}$, 
implying that 
$x_{\alpha}\in\rcohlgy{\partial\LL_{I}}$ is not an elment of 
$\cohlgy{\mathcal{Z}_{\partial\LL-\{1\}}}$ while 
$y_{\alpha}\in\rcohlgy{\partial\LL_{J}}$ is. That is, $f^{\ast}(x_{\alpha})=0$ 
while $f^{\ast}(y_{\alpha})\neq 0$. Similarly, if $1\in J_{\alpha}$ 
then $f^{\ast}(x_{\alpha})\neq 0$ and $f^{\ast}(y_{\alpha})=0$. 
\end{proof}

\section{Panov's problem} 
\label{sec:panov} 

Recall from the Introduction that if $K$ is $\ell$ disjoint points 
then there is a homotopy equivalence 
\begin{equation} 
   \label{ptequiv} 
   \zk\simeq\bigvee_{k=3}^{\ell+1}  (S^{k})^{\wedge (k-2)\binom{\ell}{k-1}} 
\end{equation} 
and if $P$ is a simple polytope of dimension $d$ obtained from $\Delta^{d}$ 
by $\ell -1$ vertex cut operations (in any order) then there is a diffeomorphism 
\begin{equation} 
   \label{connequiv} 
   \mathcal{Z}(P)\cong 
          \#_{k=3}^{\ell+1}  (S^{k}\times S^{\ell+2d-k})^{\#(k-2)\binom{\ell}{k-1}}. 
\end{equation} 
Panov posed the problem of identifying the nature of the correspondence 
between the decompositions in~(\ref{ptequiv}) and~(\ref{connequiv}). 
In this section we give an answer to the problem.  

Let $P$ be a simple polytope of dimension $d$ which has been 
obtained from $\Delta^{d}$ by $\ell -1$ vertex cuts. Dualizing, 
$P^{\ast}$ is a stacked polytope of dimension $d$ with $\ell$ 
stacks. By Proposition~\ref{stackindep}, the homotopy type of 
$\mathcal{Z}(P)=\mathcal{Z}_{\partial P^{\ast}}$ is independent 
of the stacking order of $P^{\ast}$. We may therefore analyze the 
homotopy type of $\mathcal{Z}(P)$ by analyzing the homotopy 
type of $\mathcal{Z}_{\partial\LL}$. 

\begin{proof}[Proof of Theorem~\ref{Panovsoln}] 
Consider the inclusion 
\[\namedright{\partial\LL-\{1\}}{}{\partial\LL}.\]  
The moment-angle complex, regarded as a polyhedral product, is natural 
for maps of simplicial complexes, so we obtain an induced map of 
moment-angle complexes 
\[f\colon\namedright{\mathcal{Z}_{\partial\LL-\{1\}}}{}{\mathcal{Z}_{\partial\LL}}.\] 
By Corollary~\ref{deletioncor} and~(\ref{connequiv}), up to homotopy 
equivalences $f$ can be regarded as a map 
\[f\colon\namedright{\bigvee_{k=3}^{\ell+1}  (S^{k})^{\wedge (k-2)\binom{\ell}{k-1}}} 
      {}{\#_{k=3}^{\ell+1}  (S^{k}\times S^{\ell+2d-k})^{\#(k-2)\binom{\ell}{k-1}}}.\] 
In general, whenever $K^{\prime}$ is a full subcomplex of $K$, by~\cite{BBCG} 
there is a retract of $\cxx^{K^{\prime}}$ off of $\cxx^{K}$. In our case,  
since $\partial\LL-\{1\}$ is a full subcomplex of $\partial\LL$, the map $f$ 
has a left homotopy inverse 
\[g\colon\namedright{\#_{k=3}^{\ell+1}  
       (S^{k}\times S^{\ell+2d-k})^{\#(k-2)\binom{\ell}{k-1}}}{} 
       {\bigvee_{k=3}^{\ell+1}  (S^{k})^{\wedge (k-2)\binom{\ell}{k-1}}}.\] 
This proves parts~(a), (b) and (c) of the theorem. Part~(d) is Lemma~\ref{cupseparate}. 
\end{proof} 

More is true than stated in Theorem~\ref{Panovsoln}, and it may be useful 
to elaborate on it. As in Section~\ref{sec:cup}, let $\mathcal{I}$ be an index 
set which runs over all the products of two spheres in the connected 
sum~(\ref{connequiv}). Each $\alpha\in\mathcal{I}$ corresponds 
to a product of spheres $S^{k}\times S^{\ell +2d-k}$ which determines 
a nontrivial cup product in $\cohlgy{\mathcal{Z}_{\partial\LL}}$: 
if $x_{\alpha},y_{\alpha}\in\cohlgy{\mathcal{Z}_{\partial\LL}}$ are generators 
corresponding to the spheres 
$S^{k}\vee S^{\ell +2d-k}\subset S^{k}\times S^{\ell +2d-k}$ 
then $x_{\alpha}\cup y_{\alpha}\neq 0$. By 
Proposition~\ref{cupseparate}, $f^{\ast}$ is nonzero for one and only 
one of $x_{\alpha}$ or $y_{\alpha}$. It is not immediately clear which 
of $x_{\alpha}$ or $y_{\alpha}$ is sent nontrivially by~$f_{\ast}$ to  
$\cohlgy{\mathcal{Z}_{\partial\LL-\{1\}}}$, so write $z_{\alpha}$ 
for the generator which has nontrivial image. By Lemma~\ref{cupprods}, 
$z_{\alpha}$ is the dual of the Hurewicz image of the composite of inclusions 
\[i_{\alpha}\colon S^{t_{\alpha}}\hookrightarrow S^{k}\vee S^{\ell+2d-k}\hookrightarrow 
       {\#_{k=3}^{\ell+1}  (S^{k}\times S^{\ell+2d-k})^{\#(k-2)\binom{\ell}{k-1}}}\] 
where $t_{\alpha}$ is $k$ or $\ell+2d-k$ depending on whether 
$z_{\alpha}$ is $x_{\alpha}$ or $y_{\alpha}$. Taking the wedge sum of 
all the maps $i_{\alpha}$ for every $\alpha\in\mathcal{I}$ we obtain a map  
\[i\colon\namedright{\bigvee_{\alpha\in\mathcal{I}} S^{t_{\alpha}}}{} 
       {\#_{k=3}^{\ell+1}  (S^{k}\times S^{\ell+2d-k})^{\#(k-2)\binom{\ell}{k-1}}}\] 
with the property that $i^{\ast}$ factors through $f^{\ast}$. That is, 
there is a commutative diagram 
\begin{equation} 
  \label{cohlgyfactor} 
  \diagram 
          & \cohlgy{\bigvee_{\alpha\in\mathcal{I}} S^{t_{\alpha}}} \\ 
          \cohlgy{\bigvee_{k=3}^{\ell+1}  (S^{k})^{\wedge (k-2)\binom{\ell}{k-1}}} 
                \urto^{\phi} 
             & \cohlgy{\#_{k=3}^{\ell+1}  (S^{k}\times S^{\ell+2d-k})^{\#(k-2)\binom{\ell}{k-1}}} 
                 \uto^{i^{\ast}}\lto^-{f^{\ast}} 
  \enddiagram 
\end{equation}  
for some ring map $\phi$. Note at this point that $\phi$ 
need not be induced by a map of spaces, it exists only on the level of cohomology. 

By construction, $i$ is the inclusion of one factor in each product of 
spheres in the connected sum. So $i^{\ast}$ is an epimorphism taking 
ring generators to ring generators. 
The commutativity of~(\ref{cohlgyfactor}) therefore implies that $\phi$ is also 
an epimorphism, and must take ring generators to ring generators. Now observe 
that both $\bigvee_{\alpha\in\mathcal{I}} S^{t_{\alpha}}$ and 
$\bigvee_{k=3}^{\ell+1}  (S^{k})^{\wedge (k-2)\binom{\ell}{k-1}}$ 
are wedges of precisely the same number of spheres. So the 
domain and range of $\phi$ have the same number of ring  
generators. Hence $\phi$ must be an isomorphism. 

Finally, we geometrically realize $\phi$. Consider 
the composite 
\[\nameddright{\bigvee_{\alpha\in\mathcal{I}} S^{t_{\alpha}}}  
      {i}{\#_{k=3}^{\ell+1}  (S^{k}\times S^{\ell+2d-k})^{\#(k-2)\binom{\ell}{k-1}}} 
      {g}{\bigvee_{k=3}^{\ell+1}  (S^{k})^{\wedge (k-2)\binom{\ell}{k-1}}}.\] 
Taking cohomology, by~(\ref{cohlgyfactor}) we obtain 
$i^{\ast}\circ g^{\ast}=\phi\circ f^{\ast}\circ g_{\ast}$. Since 
$g$ is a left homotopy inverse of~$f$, we therefore have 
$i^{\ast}\circ g^{\ast}=\phi$. Since $\phi$ is an isomorphism, 
so is $i^{\ast}\circ g^{\ast}$, implying by Whitehead's Theorem 
that $g\circ i$ is a homotopy equivalence. Thus $\phi$ is 
the map induced in cohomology by the homotopy equivalence 
$g\circ i$. Note, however, that it may not be the case that 
there is a homotopy $i\simeq f\circ(g\circ i)$, that is, it may 
not be the case that~(\ref{cohlgyfactor}) can be improved to a homotopy 
commutative diagram on the level of spaces.

\bibliographystyle{amsalpha}

\end{document}